\documentclass[11pt,
twoside,%
onecolumn,%
nofloats,%
nobibnotes,%
nofootinbib,%
superscriptaddress,%
noshowpacs,%
centertags]%
{revtex4}
\usepackage{ljm1}
\usepackage{enumerate}
\usepackage{upgreek}

\usepackage{amsfonts,amssymb,mathrsfs,amscd}
\usepackage{amsmath,latexsym}
\usepackage{amsthm}
\usepackage{verbatim}
\usepackage{eucal}
\usepackage{graphicx}
\usepackage{enumerate}
\usepackage{xcolor}
\usepackage{amsfonts,mathrsfs,amscd} 
\usepackage{amssymb,latexsym}
 \usepackage{amsmath} 
\usepackage{enumerate} 
\usepackage{verbatim}

\usepackage[pdftex,unicode,colorlinks,linkcolor=blue,
citecolor=red,bookmarksopen,pdfhighlight=/N]{hyperref}

\usepackage{amsbib}

\newtheorem*{thW}{Weierstrass Theorem}

\newtheorem{proposition}{Proposition} 

\newcommand{\RR}{\mathbb{R}} 
\newcommand{\CC}{\mathbb{C}} 
\newcommand{\NN}{\mathbb{N}}

\DeclareMathOperator{\sbh}{{\sf sbh}} 
\DeclareMathOperator{\supp}{supp} 
\DeclareMathOperator{\Meas}{{\sf Meas}} 
\DeclareMathOperator{\dd}{\,{\rm d}}

\DeclareMathOperator{\dist}{{\sf dist}} 
\DeclareMathOperator{\clos}{{\sf clos}} 
\DeclareMathOperator{\intr}{{\sf int}}
\DeclareMathOperator{\har}{{\sf har}} 
\DeclareMathOperator{\Hol}{{\sf Hol}} 
\DeclareMathOperator{\Zero}{{\sf Zero}}

\DeclareMathOperator{\comp}{cmp}

\DeclareMathOperator{\up}{\text{\rm \tiny up}}
 \DeclareMathOperator{\lw}{\text{\rm \tiny low}}

\begin{document}
\titlerunning{Necessary and sufficient conditions for zero subsets}
\authorrunning{B. N. Khabibullin, F. B. Khabibullin}

\title{Necessary and sufficient conditions\\ for zero subsets  of holomorphic functions}

\author{\firstname{B. N.}~\surname{Khabibullin}}
\email[E-mail: ]{khabib-bulat@mail.ru } 
\affiliation{Bashkir State University, Zaki Validi Str. 32, Bashkortostan, 420076 Russian Federation}

\author{\firstname{F. B.}~\surname{Khabibullin}}
\email[E-mail: ]{khabibullinfb@mail.ru} 
\affiliation{Ufa State Aviation Technical University, Karl Marx Str. 12, Bashkortostan, 420008 Russian Federation}


\received{December 14, 2020}

\begin{abstract} Let $D$ be a domain in the complex plane,  $M$ be an extended real function on $D$. 
If $f$ is a non-zero holomorphic  function on $D$ with an upper constraint $|f|\leq \exp M$ on this domain $D$, 
 then it is natural to expect that there must be some upper constraints on the distribution of zeros of this holomorphic function
exclusively in terms of the function $M$ and the geometry of the domain $D$.  We have investigated this question in detail in our previous works in the case when  $M$ is a subharmonic function and the domain $D$ is arbitrary or with a non-polar boundary. 
 The answer was given in terms of limiting the distribution of zeros of $f$ from above via the Riesz measure of the subharmonic function $M$. In this article, the function $M$ is the difference of subharmonic functions, or a $\delta$-subharmonic function, 
and the upper constraints are given in terms of the Riesz charge of this $\delta$-subharmonic function $M$. 
These results are also new to a certain extent for the subharmonic function $M$. 
 The case when the domain D is the complex plane is considered separately.  For the complex plane, it is possible to reach the criterion level. 

\end{abstract}

\subclass{30C15, 31A05, 30D20}

\keywords
{\it  holomorphic function, zero set, entire function, subharmonic function, Riesz measure, $\delta$-subharmonic function, Jensen measure}

\maketitle

\section{Introduction}\label{s10}

 As usual, $\mathbb N:=\{1,2, \dots\}$,  $\RR$, ${\mathbb{R}}^+:=\bigl\{x\in \RR \bigm| x\geq 0 \bigr\}$, $\CC$ are  sets of all {\it natural, real, positive, complex\/} numbers, respectively. 
We often denote singleton sets $\{x\}$ by its inner symbol $x$ without curly braces.
So,  $\NN_0:=0\cup \NN$, and  ${\mathbb{R}}^+\setminus 0$ is the set of all {\it strictly positive numbers,\/}
$\overline \RR:=-\infty\cup \RR\cup +\infty$ is the {\it extended real line,\/} $\overline \RR^+:=\RR^+\cup +\infty$, and 
$\CC_{\infty}:=\CC\cup \infty$ is the  {\it extended complex plane.\/}
Each this set is endowed with its natural order, algebraic, geometric and topological structure with modulus $|z|$ and conjugate number $\Bar z$ for $z\in \CC$, and $\inf \varnothing:=+\infty:=\sup \RR$, $\sup \varnothing:=-\infty:=\inf  \RR$ for the {\it empty set\/} $\varnothing$. 
For a subset $S\subset \CC_{\infty}$, $\Hol(S)$ is the algebra over $\CC$ of all {\it holomorphic functions\/} $f$ on an open subset $O_f\supset S$  in $\CC_{\infty}$. So,  $\Hol(\CC)$ is the algebra of all {\it entire functions.\/}

Let $S$ be a set of points,  $J$ be a {\it index set,\/} ${\sf Z}:=\{{\sf z}_{   j}\}_{   j\in J}$ be  {\it indexed  set\/}
of points ${\sf z}_j\in S$. These indexed sets will be called the {\it distribution of points\/} in $S$. 
We write  $z\in {\sf Z}$ if there is ${\sf z}_j=z$, and ${\sf Z}\subset S'\subset S$ if  ${\sf z}_ j\in
 S'$ for each $ j\in J$.  
The {\it counting function\/} $n_{\sf Z}\colon S\to \overline \NN$ of $\sf Z$
 is defined by
\begin{subequations}\label{n}
\begin{align}
n_{\sf Z}(z)&:=
\sum_{{\sf z}_ j=z}1 \in \overline \NN_0:=\NN_0\cup+\infty\quad\text{at each $z\in S$.}
\tag{\ref{n}z}\label{df:divz}\\
\intertext{We denote by the same symbol $n_{\sf Z}$ the \textit{counting measure\/}  $n_{\sf Z}\colon 2^S\to \overline \NN_0$ of ${\sf Z}$ defined by}
n_{\sf Z}(S')&\overset{\eqref{df:divz}}{:=}\sum_{{\sf z}_ j\in S'}1 =
\sum_{z\in S'}n_{\sf Z}(z)\in 
\overline \NN\quad\text { for each $S'\subset S$.}
\tag{\ref{n}n}\label{df:divmn}
\end{align}
\end{subequations}
We say that distributions ${\sf Z}$ and ${\sf Z'}$ of points in $S$ are equal and write ${\sf Z}={\sf Z'}$ if $n_{\sf Z}=n_{{\sf Z}'}$ on $S$.  Thus, both the counting function and the counting measure uniquely determine distributions of points in $S$. We write ${\sf Z}\subset {\sf Z}'$ if $n_{\sf Z}\leq n_{{\sf Z}'}$ on $S$.

Let $O\neq \varnothing$ be an open subset in $\CC$. The {\it zero set\/}  of $f\in \Hol(O)$ is a distribution   $\Zero_f$ of points in $O$ with  the  {\it counting function\/} of the {\it multiplicity of zeros\/} of $f$ defined by
\begin{equation*}
n_{\Zero_f}(z)
\underset{\text{\tiny $z\in O$}}{:=} \sup\left\{p\in \NN_0\Bigm| \limsup_{z\neq z'\to z}\frac{\bigl|f(z)\bigr|}{|z' -z|^p}<+\infty\right\}. 
\end{equation*} 
So, for \textit{the zero function\/} $0\in \Hol(O)$, we have  $n_{\Zero_0}(z)\overset{\eqref{df:divz}}{=}+\infty$ for each $z\in  O$ and $n_{\Zero_0}(S)\overset{\eqref{df:divmn}}{=}+\infty$ for each $S\subset   O$.
The counting function $n_{\Zero_f}$ of $f\in \Hol(O)$ is often called the {\it divisor of zeros\/} of $f$.
If ${\sf Z}\subset \Zero_f$, then we write $f({\sf Z})=0$ and say that  ${\sf Z}$ is a {\it zero subset\/} of $f$ and
$f$ \textit{vanishes on\/}  ${\sf Z}$.

For a subset $S\subset \CC_{\infty}$ we let 
$\complement S:=\CC_{\infty}\setminus S$, 
$\clos S$, $\intr S:=\complement (\clos \complement  S)$, and $\partial S:=\clos S\setminus\intr S$ denote its
 {\it complement,\/} {\it closure,} {\it interior,} and {\it boundary\/}  always in $\CC_{\infty}$.
We write $S\Subset O$ if $\clos S\subset O$. A distribution of points ${\sf Z}\subset O$ is {\it local finite\/} if
$n_{\sf Z}(S)<+\infty$ for each $S\Subset O$.

\begin{thW} Let ${\sf Z}$ be a distribution of points in an open set $O\subset \CC$.
  The following three statements are equivalent:
\begin{itemize}
\item  There is  $f\in \Hol(O)$ such that  $f\neq 0$ and  ${\sf Z}=\Zero_f$.
\item  There is  $f\in \Hol(O)$ such that  $f\neq 0$ and  $f({\sf Z})=0$.
\item  $\sf Z$ is local finite.
\end{itemize}
\end{thW}

Throughout this article, we consider only locally finite distributions ${\sf Z}$ of points in an open connected subset $D\subset \CC$,
i.e. $D$ is a {\it domain} in the complex plane $\CC$.

If an additional constraint $\ln |f|\leq M$ on $D$ is imposed on a holomorphic function $f\neq 0$, where $M\colon D\to \overline \RR$ is an {\it extended real function,\/} then the problem of describing zero sets and zero subsets becomes much more complicated. In particular, zero sets are very often not the same as zero subsets.  In our paper, we consider only zero subsets under a restriction from above of the form $\ln |f|\leq M$, when $M$ is the difference of subharmonic functions with the Riesz charge  $\varDelta_M$.  For subharmonic functions $M$ \cite{Rans}--\cite{Helms}, this question was considered earlier in our series of works \cite{KhaRoz18}--\cite{MenKha20}. In this article, we consider the difference  $M=M_{\up}-M_{\lw}$ of subharmonic functions  
$M_{\up}$ and $M_{\lw}$ with the Riesz measures $\varDelta_{M_{\up}}$ of $M_{\up}$
and $\varDelta_{M_{\lw}}$ of $M_{\lw}$ and with the Riesz charge $\varDelta_M:=\varDelta_{M_{\up}}-\varDelta_{M_{\lw}}$ of $M$. It is shown that if $\ln |f|\leq M$, $f\neq 0$ and $f({\sf Z})=0$, then there is a number $C\in\RR^+$ such that 
\begin{equation}\label{1}
\sum_{{\sf z}_j}v({\sf z}_j)\leq \int v \dd \varDelta_M+C 
\end{equation} 
when  $v$  runs through a very wide class of  test positive functions.  These are the necessary conditions from Sec. \ref{Sec2}, Theorem \ref{thn}. Conversely, if \eqref{1} is satisfied for much more narrower subclasses of test positive smooth subharmonic 
functions $v$, then we have almost converse statements in the form 
$\ln |f|\leq M^{\odot}+R$, where $M^{\odot}$ are  certain averaging of $M$ over small disks,   
and $R$ is a small addition related to the distance to  the boundary $\partial D$ of $D$. 
These are sufficient conditions from Sec. \ref{Sec3}, Theorem \ref{th0}.
 In Sec.~\ref{Sec3}, we consider domains $D$ with  {\it non-polar} boundary  $\partial D$,  or, equivalently, 
with non-polar complement\/ $\CC\setminus D$. This very broad class of domains $D$ includes most of the  considered in function theory and its applications. So, if $\CC\setminus D$ contains a  connected subset with more than one point, then the boundary $\partial D$ is non-polar \cite[Corollary 3.6.4]{Rans}, \cite[Theorem 5.12]{Helms}. 
But  the boundary $\partial \CC=\infty\in \CC_{\infty}$ of $\CC$ is polar.  
For the case of $D=\CC$ in Sec.~\ref{Sec4}, Theorem \ref{th3}, we obtain a criterion for a wide class of functions $M$ 
under a single condition: {\it there are numbers $P,C\in \RR^+$ such that\/}
\begin{equation}\label{M0}
\int_0^{2\pi} M_{\up}\Bigl(z+\frac{1}{(1+|z|)^P}e^{i\theta}\Bigr)\dd \theta\leq M_{\up}(z)+C
\quad\text{for each $z\in \CC$.}
\end{equation}

\section{Necessary conditions for zero subsets  of holomorphic functions  with upper constraints in domain}\label{Sec2}

\subsection{Subharmonic functions and measures}

Let  $\sbh(S)$ be the cone over $\RR^+$ of all {\it subharmonic functions\/} $u$ on open neighborhoods of $S\subset \CC_{\infty}$ including  $(-\infty)$-function $\boldsymbol{-\infty}\colon z\underset{z\in S}{\longmapsto} -\infty$, 
and $\har (S):=\sbh(S)\cap\bigl(-\sbh(S)\bigr)$   be the space over $\RR$ 
 of all {\it harmonic functions\/} $u$ on $S$. We denote by
\begin{equation*}
\sbh_*( D):=\bigl\{u\in \sbh( D)\bigm| u\neq\boldsymbol{-\infty}\bigr\}
\end{equation*}
the class of all {\it nontrivial\/} subharmonic functions on a domain  $D\subset \CC$.
 
If  $u\in \sbh_*( D)$, then its {\it Riesz measure\/}  is denoted by
\begin{equation*}
\varDelta_u:= \frac{1}{2\pi}{\bigtriangleup}  u\in \Meas^+(D), 
\end{equation*}
where ${\bigtriangleup}$  is  the {\it Laplace operator\/}  acting in the sense of the  theory of distributions or generalized functions, and $\Meas^+(D)$ is the cone over $\RR^+$ of {\it positive Radon measures\/} on $D$. 
But by definition, the Riesz measure of $(-\infty)$-function $\boldsymbol{-\infty}$ on $D$  is such that
$\varDelta_{\boldsymbol{-\infty}}(S)=+\infty$ for  each $S\subset  D$. 

If $f\in \Hol (D)$, then $\ln |f|\in \sbh(D)$ and 
$n_{\Zero_f}=\varDelta_{\ln |f|}$ \cite[3.7.8]{Rans}.

Given  $f\in F\subset {\overline{\RR}}^X $, we set $f^+\colon x\mapsto \max \{0,f(x)\}$, and  $F^+:=\{f\in F \colon f=f^+\}$.
For a  sequence  $(f_k)_{k\in \NN}\subset \overline \RR^X$, we write $f_k\underset{k\to \infty}{\nearrow} f$ if
this sequence $(f_k)_{k\in \NN}$ is  increasing and
$f=\lim\limits_{k\to \infty}f_k$;    
\begin{equation}\label{Fuparrow}
F^{\uparrow}:=\biggl\{ f\in \overline{\RR}^X \Bigm| \;\exists (f_k)\subset F, \;  f_k\underset{k\to \infty}{\nearrow} f\biggr\}.
\end{equation}
So, $\sbh^+(S):=\bigl(\sbh(S)\bigr)^+$ is the class of all positive subharmonic functions on $S$.
Let $S\Subset D$ be a closed subset.   We define 
 classes  \cite[Definition 1]{KhaRoz18}, \cite{KhaKha19} 
\begin{equation}\label{v}
\sbh_0^+(D\setminus S; \leq b):=\Bigl\{v\in \sbh^+(D\setminus S)\Bigm|
 v\leq b \text{ on  $D\setminus S$},\quad 
\limsup_{D\ni z'\to z} v(z')\underset{\text{\tiny $z\in \partial D$}}{=}0\Bigr\}.
\end{equation}
of  {\it  test subharmonic positive functions $v$ for $D$ outside of $S$  with an upper bound of $b\in \RR^+$,}
and 
\begin{equation*}
\sbh_0^{+\uparrow}(D\setminus S; \leq b)\overset{\eqref{Fuparrow}}{:=}\bigl(\sbh_0^+(D\setminus S; \leq b)\bigr)^{\uparrow}
\end{equation*}
of  {\it test upper positive  functions $v$ for $D$ outside of $S$  with an upper bound of $b\in \RR^+$.\/}

The difference of two nontrivial subharmonic functions is called a {\it nontrivial $d$-subharmonic function,\/} or a 
{\it nontrivial $\delta$-subharmonic function.}  A {\it $d$-subharmonic function\/} \cite{Arsove53}, \cite[3.1]{KhaRoz18}, \cite{Az}
\begin{equation}\label{M}
M:=M_{\text{\tiny \rm up}}-M_{\text{\tiny \rm low}}, \quad M_{\text{\tiny \rm up}}\in \sbh_*(D), 
\quad M_{\text{\tiny \rm low}}\in \sbh_*(D),
\quad \varDelta_M:=\varDelta_{M_{\text{\tiny \rm up}}}-\varDelta_{M_{\text{\tiny \rm low}}},
\end{equation}
with the { \it  Riesz charge\/} $\varDelta_M$ of $M$ is 
defined at each point $z\in D$ where  $M_{\text{\tiny \rm low}}(z)\neq -\infty$. 
Below we put $M(z):=+\infty$ if  $M_{\text{\tiny \rm low}}(z)= -\infty$. 

 For a Borel subset $S\subset \CC$, we denote by $\Meas(S)$  the class of all Borel signed measures, or \textit{charges,\/} on $S$.   $\Meas_{\comp}(S)$ is the class of charges $\mu \in \Meas(S)$ with a compact support $\supp \mu \Subset S$; 
\begin{equation*}
\Meas^+(S):=( \Meas (S))^+, \quad \Meas_{\comp}^+(S):=( \Meas_{\comp} (S))^+.  
\end{equation*}
For a charge $\mu \in \Meas(S)$, we let
$\mu^+$, $\mu^-:=(-\mu)^+$ and $|\mu| := \mu^+ +\mu^-$ respectively denote its {\it upper, lower,\/} and {\it total variations.}
A function $f\colon S\to \overline \RR$ is {\it $\mu$-integrable\/}  if there are four integrals $\int f^{\pm}\dd\mu^{\pm} \in \overline \RR^+$
such that  $\int f^{+}\dd\mu^{+}+\int f^{-}\dd\mu^{-}<+\infty$ for $\int f^{-}\dd\mu^{+}+\int f^{+}\dd\mu^{-}=+\infty$, and vice versa, $\int f^{-}\dd\mu^{+}+\int f^{+}\dd\mu^{-}<+\infty$ for  $\int f^{+}\dd\mu^{+}+\int f^{-}\dd\mu^{-}=+\infty$. 
A $\mu$-integrable function $f\colon S\to \overline \RR$ is {\it $\mu$-summable\/}  if  $\int|f|\dd|\mu| \neq +\infty$. 

\subsection{Jensen measures and its potentials} 

A measure $\mu \in \Meas_{\comp}^+(D)$ is a {\it Jensen measure for domain $D\subset D$ at $z_0\in D$\/} if \cite{Kha03}, 
\cite{KhaRoz18}
\begin{equation}\label{Jm}
u(z_0)\leq \int u\dd \mu\quad\text{for each $z_0\in \sbh(D)$}.
\end{equation} 
We denote by ${\sf J}_{z_0}(D)$ the class of all these Jensen measures. Obviously,  $\mu(D)=1$  for every $\mu\in {\sf J}_{z_0}(D)$.
For $\mu\in {\sf J}_{z_0}(D)$, the function   
\begin{equation}\label{Vmeas}
V_{\mu}\colon z\longmapsto \int \ln |z'-z|\dd \mu-\ln |z| , \quad z\in \CC\setminus z_0,\quad V(\infty):=0,
\end{equation}
is the {\it logarithmic potential of $\mu\in {\sf J}_{z_0}(D)$ with pole $z_0$ for $D$.} 

A positive subharmonic function $V\in \sbh^+(\CC_{\infty}\setminus z_0)$ 
is a {\it Jensen potential for $D$ with pole $z_0\in D$\/} if there is a subset $S_V\Subset D$ 
such that $V(z)=0$ for each  $z\in D\setminus S_V$ and
  \begin{equation*}
\limsup_{z_0\neq z\to z_0}\frac{V(z)}{-\ln |z-z_0|}\leq 1.
\end{equation*}
We denote by ${\sf PJ}_{z_0}(D)$ the class of all these Jensen potentials. 
If $D'\Subset D$ be a subdomain in domain $D\subset \CC$, then  its  {\it Green's function\/} 
${\sf g}_{D'}(\cdot, z_0)$ with pole $z_0\in D'$ belongs to ${\sf PJ}_{z_0}(D)$.

\begin{proposition}[{\cite[Propositions 1.2, 1.4]{Kha03}}]\label{pr1} Let $D\neq \varnothing$ be a domain  in $\CC$ and $z_0\in D$. 
\begin{enumerate}[{\rm (i)}]
\item\label{im}    The mapping 
$\mathcal P\colon \mu \overset{\eqref{Vmeas}}{\longmapsto} V_{\mu} $ is an affine  bijection from ${\sf J}_{z_0}(D)$ on\/ ${\sf PJ}_{z_0}(D)$, and 
\begin{equation}\label{P-1}
\mathcal P^{-1}(V)=\varDelta_V\bigm|_{\CC\setminus z_0}+\biggl(1-\limsup_{z_0\neq z\to z_0}\frac{V(z)}{-\ln |z-z_0|}\biggr)\delta_{z_0} , 
\quad V\in {\sf PJ}_{z_0}(D).
\end{equation}
where $\delta_{z_0}$ is the  {\it Dirac probability measure\/} with the {\it support\/} $\supp \delta_{z_0}=z_0$.
 \item\label{iim} If $\mu\in {\sf J}_{z_0}(D)$, then the following Poisson\,--\,Jensen formula holds:
 \begin{equation}\label{PJf}
u(z_0)=\int_{D\setminus z_0}  u\dd \mu-\int_DV_{\mu}\dd \varDelta_u
\quad\text{for each $u\in \sbh(D)$ with $u(z_0)\neq -\infty$.}
\end{equation} 
\end{enumerate}
\end{proposition}  
We denote by ${\sf PJ}_{z_0}^{\uparrow}(D):=\bigl({\sf PJ}_{z_0}(D)\bigr)^{\uparrow}$ the class of
all {\it test Jensen functions.\/} 

If $D\subset \CC$ is a domain with non-polar boundary $\partial D$, then there is its  {\it Green's function\/} 
${\sf g}_D(\cdot, z_0)$ with pole $z_0\in D$, and ${\sf g}_D(\cdot, z_0)$  is 
the largest Jensen test function in ${\sf PJ}_{z_0}^{\uparrow}(D)$. 
If the boundary $\partial D$ is polar, then the largest test Jensen function in ${\sf PJ}_{z_0}^{\uparrow}(D)$ is 
the $(+\infty)$-function $\boldsymbol{+\infty}\colon z\underset{z\in D}{\longmapsto} +\infty$.
   
\subsection{Main result on necessary conditions for zero subsets in domains}

The main task of this section is to establish the largest possible range of necessary conditions for the distribution of zero subsets of holomorphic functions  $f\in \Hol(D)$ satisfying the upper constraint $\ln |f|\leq M$ on $D$. 
We establish these conditions for arbitrary domains in $\CC$ and arbitrary $d$-subharmonic majorants $M$ from \eqref{M}.

\begin{theorem}[{necessary conditions}]\label{thn} 
Let ${\sf Z}$ be a locally finite distribution of points in a domain $D\subset \CC$
and  let $M$ be a function \eqref{M}. 
Suppose that 
there exists a function $f\in \Hol(D)$ such that $f\neq 0$, $f({\sf Z})=0$ and\/  
\begin{equation}\label{q}
\ln \bigl|f(z)\bigr| \leq M(z) \quad\text{at  each $z\in D$}.
\end{equation}
Then, for any closed set $S\Subset D$  with\/ $\intr  S\neq \varnothing$ and for any $b\in \RR^+$, there is a number  $C\in \RR^+$ such that, 
for each test upper positive  function $v\overset{\eqref{v}}{\in} \sbh_0^{+\uparrow}(D\setminus S; \leq b)$, we have 
\begin{equation}\label{V}
\sum_{{\sf z}_j\in D\setminus S}v({\sf z}_j)\leq  \int_{D \setminus S}v\dd \varDelta_M +C
\quad\text{provided $v$ is $\varDelta_M$-summable on $D\setminus S$.}
\end{equation} 
If $z_0 \in D$, $f(z_0)\neq 0$ and $M_{\up}(z_0)+ M_{\lw}(z_0)\neq -\infty$,
then there is  $C\in \RR$ such that  \eqref{V} holds with the singleton $S:=\{z_0\}$ 
for each test  Jensen function $v\in {\sf PJ}_{z_0}^{\uparrow}(D)$.    
\end{theorem}

\begin{proof} In the case $v\overset{\eqref{v}}{\in} \sbh_0^{+}(D\setminus S; \leq b)$ without $\uparrow$ we use 
\begin{lemma}[{\cite[Main Theorem]{KhaRoz18}}]\label{lemmt}
For any point $z_0\in \intr S$ satisfying $M_{\up}(z_0)+ M_{\lw}(z_0)\neq -\infty$,  any regular (for the Dirichlet problem) domain
$\tilde D\subset \CC$ with the Green function ${\sf g}_{\tilde D}( \cdot, z_0)$
with a pole at $z_0\in \tilde D$ which satisfies the conditions $S\Subset \tilde D \subset D$ and $\CC_{\infty} \setminus \clos \tilde D\neq \varnothing$, any function $u\in \sbh_*(D)$  satisfying the inequality $u\leq M$ on $D$, and any test function $v\in \sbh_0^{+}(D\setminus S; \leq b)$ 
the following inequality holds:
\begin{equation}\label{uM}
Cu(z_0)+\int_{D\setminus S}v\dd \varDelta_u\leq \int_{D\setminus S} v\dd \varDelta_M+
\int_{\tilde D\setminus S} v\dd \varDelta_{M_{\lw}}^-+C\overline C_M ,
\end{equation}
where $C:=b/\inf\limits_{z\in \partial S}{\sf g}_{\tilde D}(z,z_0)>0$,
and the value $+\infty$ is possible for
\begin{equation}\label{gDM}
\overline C_M:=\int_{\tilde D\setminus z_0}{\sf g}_{\tilde D}(\cdot, z_0)\dd \varDelta_M
+\int_{\tilde D\setminus S}{\sf g}_{\tilde D}(\cdot, z_0)\dd \varDelta_M^-+M^+(z_0),
\end{equation} 
but for  $\tilde D\Subset D$, 
this is a certain constant $\overline C_M<+\infty$ independent  of $v$ and $u$.
\end{lemma}
We put $u\overset{\eqref{q}}{:=}\ln |f|$ and  choose a point $z_0\in \intr S$  and a  regular domain $\tilde D\Subset D$ 
with $f(z_0)\neq 0$ and the required properties in Lemma \ref{lemmt}. Then, in view of  \eqref{uM}--\eqref{gDM}, $C\in \RR$, 
$\overline C_M\in \RR$, and  
\begin{equation*}
\int_{\tilde D\setminus S} v\dd \varDelta_{M_{\lw}}^-\leq b\varDelta_{M_{\lw}}^-\bigl(\tilde D\bigr)<+\infty
\end{equation*}
are independent of $f$ and $v$. Thus, there is a number $\tilde C$ such that 
\begin{multline*}
\sum_{{\sf z}_j\in D\setminus S}v({\sf z}_j)=\int_{D\setminus S}v\dd n_{\sf Z}\leq 
\int_{D\setminus S}v\dd n_{\Zero_f}=
\int_{D\setminus S}v\dd \varDelta_{\ln|f|}\\= \int_{D\setminus S}v\dd \varDelta_u\leq \int_{D\setminus S} v\dd \varDelta_M+\tilde C-C\ln
\bigl|f(z_0)\bigr|
\end{multline*}
for each  test function $v\in \sbh_0^{+}(D\setminus S; \leq b)$.  Hence, for  $C':=\tilde C-C\ln \bigl|f(z_0)\bigr| \in \RR$, we obtain 
\begin{equation}\label{Nf}
\sum_{{\sf z}_j\in D\setminus S}v({\sf z}_j)+\int_{D\setminus S}v\dd \varDelta_{M_{\lw}}\leq \int_{D\setminus S} v\dd \varDelta_{M_{\up}}+C'
\end{equation}
for each  test function $v\in \sbh_0^{+}(D\setminus S; \leq b)$. 
Let $(v_k)_{k\in \NN}\subset  \sbh_0^{+}(D\setminus S; \leq b)$ be a increasing sequence and  $v:=\lim\limits_{k\to \infty}v_k\in \sbh_0^{+\uparrow}(D\setminus S; \leq b)$ is $\varDelta_{M}$-summable. Then 
\begin{multline*}
\int_{D\setminus S}v_k\dd (n_{\sf Z}+\varDelta_{M_{\lw}})=\sum_{{\sf z}_j\in D\setminus S}v_k({\sf z}_j)+\int_{D\setminus S}v_k\dd \varDelta_{M_{\lw}}\\
\overset{\eqref{Nf}}{\leq} \int_{D\setminus S} v_k\dd \varDelta_{M_{\up}}+C'
\leq \int_{D\setminus S} v\dd \varDelta_{M_{\up}}+C'<+\infty \quad\text{for each $k\in\NN$}. 
\end{multline*}
Applying the  monotone convergence theorem for integrals to the left-hand side, we obtain 
\begin{equation*}
\sum_{{\sf z}_j\in D\setminus S}v({\sf z}_j)+\int_{D\setminus S}v\dd \varDelta_{M_{\lw}}=\int_{D\setminus S}v\dd (n_{\sf Z}+\varDelta_{M_{\lw}})
\leq \int_{D\setminus S} v\dd \varDelta_{M_{\up}}+C' 
\end{equation*}
for each  test upper function $v\in \sbh_0^{+\uparrow}(D\setminus S; \leq b)$. 

It remains to consider the case  of test Jensen functions.  

By the Weierstrass Theorem there are  a function $f_{\sf Z}\in \Hol(D)$ with zero set $\Zero_{f_{\sf Z}}={\sf Z}$
and a  function $g\in \Hol(D)$ such that $f_{\sf Z}(z_0)\neq 0$, $g(z_0)\neq 0$ and  
$\ln |f_{\sf Z}|+M_{\lw}\overset{\eqref{q}}{\leq} M_{\up}-\ln |g|$ on $D$ 
outside some polar set, and hence everywhere.  Integrating with respect to a Jensen measure $\mu\overset{\eqref{Jm}}{\in} J_{z_0}(D)$, we obtain
\begin{equation*}
\int_D\ln |f_{\sf Z}|\dd \mu +\int_D M_{\lw}\dd \mu\leq \int_DM_{\up}\dd \mu-\int_D\ln|g|\dd \mu
\leq \int_DM_{\up}\dd \mu -\ln\bigl|g(0)\bigr|. 
\end{equation*}
By the Poisson\,--\,Jensen formula  of Proposition \ref{pr1}\eqref{iim}, \eqref{PJf}, 
to  $\ln|f|$, $M_{\up}$, $M_{\lw}$,  
we have
\begin{equation*}
\int_DV_{\mu}\dd n_{\sf Z}+\ln \bigl|f(z_0)\bigr|+\int_DV_{\mu}\dd \varDelta_{M_{\lw}}+M_{\lw}(z_0)
\leq  \int_DV_{\mu}\dd \varDelta_{M_{\up}}+M_{\up}(z_0)
 -\ln\bigl|g(0)\bigr| 
\end{equation*}
Hence 
\begin{equation*}
\int_DV_{\mu}\dd n_{\sf Z}+\int_DV_{\mu}\dd \varDelta_{M_{\lw}}
\leq  \int_DV_{\mu}\dd \varDelta_{M_{\up}}+\underset{C}{\underbrace{\Bigl(M_{\up}(z_0)-\ln \bigl|f(z_0)\bigr|-M_{\lw}(z_0)-\ln\bigl|g(0)\bigr|\Bigl)}}
 \end{equation*}
for logarithmic potentials $V_{\mu}$ of all Jensen measures $\mu\in {\sf J}_{z_0}(D)$.
By Proposition \ref{pr1}\eqref{im}, if $\mu$ runs through  ${\sf J}_{z_0}(D)$, then  $V_m$ 
runs through the class  ${\sf PJ}_{z_0}$. Thus, 
\begin{equation*}
\int_DV\dd \bigl(n_{\sf Z}+ \varDelta_{M_{\lw}}\bigr)=\int_DV\dd n_{\sf Z}+\int_DV\dd \varDelta_{M_{\lw}}
\leq  \int_DV\dd \varDelta_{M_{\up}}+C
\quad\text{for each $V\in {\sf PJ}_{z_0}(D)$}.
 \end{equation*}
Let $(V_k)_{k\in \NN}\subset  {\sf PJ}_{z_0}(D)$ be  increasing and  $V:=\lim\limits_{k\to \infty}V_k\in {\sf PJ}_{z_0}^\uparrow(D)$ be $\varDelta_{M}$-summable. Then 
\begin{equation*}
\int_DV_k\dd \bigl(n_{\sf Z}+ \varDelta_{M_{\lw}}\bigr)\leq  \int_DV_k\dd \varDelta_{M_{\up}}+C
\leq  \int_DV\dd \varDelta_{M_{\up}}+C<+\infty.
 \end{equation*} 
Applying the  monotone convergence theorem for integrals to the left-hand side, we obtain 
\begin{equation*}
\sum_{{\sf z}_j\in D\setminus S}V({\sf z}_j)+\int_{D\setminus S}V\dd \varDelta_{M_{\lw}}=\int_{D\setminus S}V\dd (n_{\sf Z}+\varDelta_{M_{\lw}})
\leq \int_{D\setminus S} V\dd \varDelta_{M_{\up}}+C' 
\end{equation*}
for each  test  Jensen function $V\in {\sf PJ}_{z_0}^{\uparrow}(D)$.    
\end{proof}

\section{Sufficient conditions for zero subsets  of holomorphic functions  with upper constraints in domains}\label{Sec3}

\subsection{Integral means of subharmonic and $d$-subharmonic functions}

We denote by $D(z,t):=\bigl\{z'\in \CC\colon |z'-z|< t\bigr\}$, $\overline D(z,t):=\bigl\{z'\in \CC\colon |z'-z|\leq  t\bigr\}$, $\partial \overline D(z,t):=\overline D(z,t)\!\setminus\!  D(z,t)$  an {\it open disk,\/} a {\it closed disk,\/} a {\it circle of radius $t\in \overline \RR^+$ centered at $z\in \CC$}, respectively.
If $D\neq \varnothing$ be  a proper domain in $\CC$, i.e. $D\neq \CC$, then  we use a function
 $r\colon D\to \RR$ on $D$ such that 
\begin{subequations}\label{r}
\begin{flalign}
\begin{cases}
0\leq  r(z)<\dist (z, \partial D):=\inf\limits_{z'\in \partial D} |z-z'| \quad \text{for each $z\in D$},\\
\inf\limits_{z\in K}r(z)>0 \quad \text{for each $K\Subset D$},
\end{cases}
\tag{\ref{r}D}\label{{r}r}\\
\intertext{but if $D=\CC$, then we  use another  function}
r(z)\underset{z\in \CC}{:=}\frac{1}{(1+|z|)^P}\quad\text{with a number  $P\in \RR^+$}.
\tag{\ref{r}$\CC$}\label{{r}P}
\end{flalign}
\end{subequations}

Let  $u\colon D\to \overline \RR$ be a function.  The {\it integral means\/} of $u$  {\it over circles} $\partial \overline D\bigl(z,r(z)\bigr)$
are denoted by 
\begin{subequations}\label{M*r0}
\begin{align}
u^{\odot r}  (z) &:=\frac{1}{2\pi}\int_0^{2\pi} u\bigl(z+r(z)e^{i\theta}\bigr)\dd \theta, 
 \quad z\in D, \quad  \overline D(z,r(z))\subset D,
\tag{\ref{M*r0}$\odot$}\label{{M*r0}o}
\\ 
\intertext{the {\it integral means\/} of $u$ {\it over  disks\/} $D\bigl(z,r(z)\bigr)$ are denoted by}
u^{\bullet r}  (z) &:=
\frac{1}{\pi r^2(z)}\int_0^{r(z)}\int_0^{2\pi} u(z+te^{i\theta})\dd \theta\,   t\!\dd t ,
\quad 
 z\in D, \quad  D(z,r(z))\subset D,
\tag{\ref{M*r0}$\bullet$}\label{{M*r0}b}
\end{align}
\end{subequations}
under the assumption that the integrals are well defined here, and  \cite[2.6]{Rans}, 
\cite{Beardon}, \cite[Theorem 3]{FM} 
\begin{equation}\label{Ml}
u\leq u^{\bullet r}\leq u^{\odot r}\leq u^{\bullet (\sqrt{e}r)}\quad\text{on $D$\/ for each  $u\in \sbh_*(D)$}, 
\end{equation}
where the last inequality is given under the assumption that $\sqrt{e}r< \dist(\cdot, \partial D)$ on $D$.

We impose one very weak requirement on function  \eqref{{r}r}.
For the  function   
\begin{subequations}\label{cup}
\begin{flalign}
\widehat r(z)&:=
\inf \biggl \{R\in \RR^+\Bigm| 
\bigcup_{z'\in D(z,r(z))} D\bigl(z',r(z')\bigr)\subset D(z, R)
\biggr\},\quad z\in D, 
\tag{\ref{cup}r}\label{{cup}r}
\\ 
\intertext{we require}
&\overline D\bigl(z,\widehat r(z)\bigr)\subset D\quad\text{for each $z\in D$.}
\tag{\ref{cup}$\rm \widehat r$}\label{{cup}barr}
\end{flalign}
\end{subequations}

We define the class \cite[(1.12)]{KhaKha19} 
\begin{multline}\label{v00}
\sbh_{00}^+(D\setminus S; \leq b):=
\Bigl\{v\in \sbh_0^+(D\setminus S; \leq b)\Bigm|\\
\text{there is a subset $S_v \Subset D$ such that $v(z)=0$ at each $z\in D\setminus S_v$}\Bigr\}.
\end{multline}
of  {\it  test subharmonic positive compactly supported functions for $D$ outside of $S\Subset D$.}

\subsection{Main result on sufficient  conditions for zero subsets in domains}

The order of formulating sufficient conditions in Theorem \ref{th0} differs from the order of formulating necessary conditions in Theorem 
\ref{thn}. First, we formulate sufficient conditions for arbitrary domains $D$ and $d$-subharmonic majorants
$M$  in terms of smooth Jensen potentials from ${\sf PJ}_{z_0}(D)$, and then we formulate sufficient conditions for arbitrary domains $D$ with non-polar boundary $\partial D$ and arbitrary $d$-subharmonic majorants $M$ from \eqref{M}
  in terms of smooth test subharmonic functions from $\sbh_0(D\setminus S;\leq 1)$.
The main task of this section is to establish the smallest possible set of sufficient conditions for the distribution of zero subsets of holomorphic functions $ f\in \Hol (D) $ satisfying the upper constraint $ \ln |f| \leq M $ on $ D $.

\begin{theorem}[{sufficient  conditions}]\label{th0}
Let ${\sf Z}$ be a locally finite distribution of points in an domain  $D\subset \CC$ containing $z_0\notin {\sf Z}$
and $M$ be a $d$-subharmonic function \eqref{M} with $M_{\up}(z_0)+M_{\lw}(z_0)\neq -\infty$. 

 If, there is a  subdomain $U_{z_0}\Subset D$ containing $z_0\in U_{z_0}$ such that 
the inequality \eqref{V} with $S:=\{z_0\}$ is fulfilled for each smooth  Jensen potential 
 \begin{subequations}\label{inP}
\begin{align}
v\in {\sf PJ}_{z_0}(D)\bigcap \har(U_{z_0}\setminus z_0)\bigcap C^{\infty}(D\setminus z_0)
\tag{\ref{inP}P}\label{Vl0}\\
\intertext{satisfying}
v(z)=-\ln |z-z_0|+O(1)\quad \text{as $z_0\neq z\to z_0$}, 
\tag{\ref{inP}$_0$}\label{l0}
\end{align}
\end{subequations}
then,  for  each function  \eqref{r} with \eqref{cup} and with a number  $P\in \RR^+$ if $D:=\CC$ and for any number $a>0$,
 there exists a function $f\in \Hol (D)$ such that $f\neq 0$,  $f({\sf Z})=0$ and 
\begin{subequations}\label{f}
\begin{align}
\ln |f|&\overset{\eqref{{M*r0}o}}{\leq} M_{\up}^{\odot \widehat r}-M_{\text{\tiny \rm low}}+R\quad \text{on $D$, where $\widehat r$
is defined in \eqref{cup} and}
\tag{\ref{f}M}\label{{f}M}
\\
 R(z)&:=\begin{cases}
\ln \dfrac{1}{r(z)}+(1+a)\ln\bigl(1+|z|\bigr)\quad\text{if $D\neq \CC$},\\
 \ln \dfrac{1}{r(z)}\quad\text{if $D\neq \CC$ is simply  connected  
or\/ $\complement \clos D\neq \varnothing$},\\
0\quad\text{if $D= \CC$},
\end{cases} \quad 
\text{at each $z\in D$.}
\tag{\ref{f}R}\label{{f}r}
\end{align}
\end{subequations}

In addition, if  the boundary $\partial D$ of $D$ is non-polar and there exist  a closed subset $S\Subset D$ with $\intr S\neq \varnothing$ and a number  $C\in \RR^+$ such that 
the inequality \eqref{V} is fulfilled for each smooth  test function $v\overset{\eqref{v00}}{\in}  \sbh_{00}^+(D\setminus S;\leq 1)\bigcap C^{\infty}(D\setminus S)$, then,  for any function   \eqref{{r}r},  there is a function $f\in \Hol (D)$ such that $f\neq 0$,  $f({\sf Z})=0$ and 
\eqref{f} is fulfilled. 
\end{theorem}

\begin{proof}[Proof of Theorem\/ {\rm \ref{th0}}] 
We first prove the statement for Jensen potentials. 
We denote by ${\sf PJ}_{z_0}^1(D)$ the class of all Jensen potentials $v$ satisfying \eqref{l0} and put (cf. \cite[(13V)]{MenKha20})
\begin{equation}\label{PJ1}
\mathcal V(U_{z_0}):={\sf PJ}_{z_0}^1(D)\bigcap \har(U_{z_0}\setminus z_0)\bigcap C^{\infty}(D\setminus z_0).
\end{equation}
 We  denote by $\Meas^+_{\infty}(D)$ the subclass of all  measures    $\mu\in \Meas^+(D)$
with a density $m\in C^{\infty}(D)$, i.e. $\dd \mu=m \dd \lambda$,  where  $\lambda$
is the Lebesgue measure on $D$, and put (cf. \cite[(13M)]{MenKha20})
\begin{equation}\label{rto}
\mathcal M(U_{z_0}):={\sf J}_{z_0}(D)\bigcap \Meas_{\comp}(D\setminus U_{z_0})\bigcap \Meas^+_{\infty}(D).
\end{equation}
 By Proposition \ref{pr1}\eqref{im}, it is easy to see  that the 
of mapping $\mathcal P^{-1}$ in \eqref{P-1} defines a bijection of subclass $\mathcal V(U_{z_0})$
on subclass $\mathcal M(U_{z_0})$ \cite[Theorem A]{MenKha20}.  

By the Weierstrass Theorem there is a function  $f_{\sf Z}\in \Hol (D)$ with zero set $\Zero_{f_{\sf Z}}={\sf Z}$. 

We have the inequality \eqref{V} with $S:=\{z_0\}$ for all Jensen potentials $v\in \mathcal V(U_{z_0})$.
Hence, by the Poisson\,--\,Jensen formula  of Proposition \ref{pr1}\eqref{iim}, \eqref{PJf}, 
applied to  $\ln|f|$, $M_{\up}$, $M_{\lw}$, and by bijection  ${\mathcal P}^{-1}(D)\colon \mathcal V(U_{z_0})
\overset{\eqref{rto}}{\longrightarrow}  \mathcal M(U_{z_0})$,
we have (cf. \cite[(15)]{MenKha20})
\begin{multline}\label{lnM}
\int_D\underset{u}{\underbrace{\bigl(\ln |f_{\sf Z}|+ M_{\lw}\bigr)}}\dd \mu=
\int_D\ln |f_{\sf Z}|\dd \mu +\int_D M_{\lw}\dd \mu 
\\
\leq \int_DM_{\up}\dd \mu
+\underset{c}{\underbrace{\Bigl(\ln\bigl|f_{\sf Z}(z_0)\bigr|+M_{\lw}(z_0)-M_{\up}(z_0)\Bigr)}}
\quad\text{for each $\mu\in \mathcal M(U_{z_0})$}.
\end{multline}

\begin{lemma}[{A very special case of \cite[Corollary 8.1.II.1]{KhaRozKha19} with $H:=\sbh_*(D)$, cf. \cite[Theorem B]{MenKha20}}]\label{lemB}
If for some number $c \in \RR$ assertion \eqref{lnM} holds, then, for any function $r$ satisfying \eqref{{r}r},
there are a function $h\in \sbh_*(D)$ and a positive function $ \check r\leq r$ from the class $C^{\infty}(D)$ such that 
\begin{equation}\label{key}
u+h\leq M_{\up}^{\circledast \check{r}}\in C^{\infty}(D) \quad\text{on $D$},
\end{equation}
where by construction {\rm \cite[(8.3--6), (8.10)]{KhaRozKha19}, \cite[(2.18--19)]{KhaKha19}}
$ M_{\up}^{\circledast \check{r}}$ are ``moving contracting'' smoothing averages over some probabilistic 
measures 
$\alpha^{(\check r(z))}\in \Meas^+_{\infty}\bigl(\overline D(z,\check r(z))\bigr)$,
obtained by the shift, compression, and normalization od a single approximate unit $a\in C^{\infty}(\CC)$, depending on the modulus $|\cdot|$
only with support $\supp a\subset \overline D(0,1)$.  
\end{lemma}
By Lemma \ref{lemB} we choose  a subharmonic function $h\in \sbh_*(D)$ such that 
\begin{equation*}
\ln |f_{\sf Z}|+ M_{\lw}+h\leq M_{\up}^{\circledast \check r}\quad \text{on $D$.}
\end{equation*}
By \cite[Proposition 3]{BaiTalKha16} or \cite[Theorem 4]{KhaLob}, we have 
$M_{\up}^{\circledast \check r}\leq M_{\up}^{\odot \check r}\leq M_{\up}^{\odot r}$ on $D$ for subharmonic function 
$M_{\up}$. Thus 
\begin{equation*}
\ln |f_{\sf Z}|+ M_{\lw}+h\leq M_{\up}^{\odot r}\quad \text{on $D$.}
\end{equation*}
Hence 
\begin{equation}\label{kk}
\ln |f_{\sf Z}|+ M_{\lw}+h^{\bullet r}\overset{\eqref{{M*r0}b}}{\leq}
\bigl(\ln |f_{\sf Z}|\bigr)^{\bullet r}+ M_{\lw}^{\bullet r}+h^{\bullet r}
\leq \bigl(M_{\up}^{\odot r}\bigr)^{\bullet r}\quad \text{on $D$.}
\end{equation}

\begin{lemma}[{\cite[Theorem 3, Corollary 3(i),(iii)]{KhaBai16}}]\label{lem2} 
Let $h\in \sbh_*(D)$ be a subharmonic function on a domain $D\subset \CC$. Then, for any number $a>0$, there is a function $g\in \Hol(D)$ such that $g\neq 0$ and 
\begin{equation}\label{ln}
\ln |g| \leq h^{\bullet t}+ R\quad\text{on  $D$,}
\end{equation}
where $R$ is a function from \eqref{{f}r}. 
\end{lemma}
By Lemma \ref{lem2} and \eqref{kk}, we get 
\begin{equation}\label{kk1}
\ln |f_{\sf Z}g|+ M_{\lw}=\ln |f_{\sf Z}|+ M_{\lw}+\ln|g|
\leq \bigl(M_{\up}^{\odot r}\bigr)^{\bullet r}+R\quad \text{on $D$,}
\end{equation}
where $f:=f_{\sf Z}g\neq 0$ and $f({\sf Z})=0$. 

The following lemma  is an elementary very special case  of \cite[Theorems 2, 4]{KhaLob}. 
\begin{lemma}\label{lem4} If   $r$ and $\widehat r$ are functions \eqref{cup}, then 
$(u^{\odot r})^{\bullet r}\leq u^{\odot \widehat r} $ on  $D$.
 \end{lemma}
By Lemma \ref{lem4}, it follows from \eqref{kk1} that  
\begin{equation*}
\ln |f|+ M_{\lw} =\ln |f_{\sf Z}g|+ M_{\lw} \leq M_{\up}^{\odot \widehat r}+R\quad\text{on $D$.}
\end{equation*}
Thus, we have proved  \eqref{f} under \eqref{V}  for the smooth Jensen potentials. 

Let us now consider the case of a domain $D$ with non-polar boundary $\partial D$. 
\begin{lemma}[{\rm \cite[Theorem 3]{KhaKha19}}]\label{lem0} Under the conditions of Theorem\/ {\rm \ref{th0},} there is  a subharmonic function $u\in \sbh_*(D)$ such that 
$n_{\sf Z}\leq \varDelta_u$ and  $ u \leq  M^{\bullet r} \quad\text{on  $D$}$. 
\end{lemma}
By Lemma \ref{lem0}  there is  a subharmonic function $u\in \sbh_*(D)$ such that 
\begin{equation}\label{in:hM}
n_{\sf Z}\leq \varDelta_u \quad\text{and}\quad     u \leq  M^{\bullet r}= M_{\up}^{\bullet r} -M_{\text{\rm \tiny low}}^{\bullet r}  
\overset{\eqref{Ml}}{\leq} M_{\up}^{\bullet r} -M_{\text{\rm \tiny low}} \quad\text{on  $D$}. 
\end{equation}
By the Weierstrass Theorem, there is   $f_{\sf Z}\in \Hol(D)$ with    $\Zero_{f_{\sf Z}}={\sf Z}$, and $f_{\sf Z}\neq 0$.

 Consider a $d$-subharmonic function $h:=u-\ln |f_{\sf Z}|$  with  Riesz charge  $\varDelta_h=\varDelta_u-n_{\sf Z}\overset{\eqref{in:hM}}{\geq} 0$, i.\,e.,  $\varDelta_h\in \Meas^+(D) $ and $h \in \sbh_*(D)$. 
It follows from \eqref{in:hM} that
 \begin{equation}\label{m1}
\ln |f_{\sf Z}|+h=u\overset{\eqref{in:hM}}{\leq}  
M_{\up}^{\bullet r} -M_{\text{\rm \tiny low}} \quad\text{on  $D$}.
\end{equation}
Hence, for $\ln |f_{\sf Z}|\in \sbh_*(D)$ and  $h\in \sbh_*(D)$, we get  
\begin{multline}\label{lnf}
\ln |f_{\sf Z}|+h^{\bullet r}\leq
\bigl(\ln |f_{\sf Z}|\bigr)^{\bullet r}+h^{\bullet r}=
\bigl(\ln |f_{\sf Z}|+h\bigr)^{\bullet r}
=u^{\bullet r}\\
\overset{\eqref{m1}}{\leq}  \bigl(M_{\up}^{\bullet r}\bigr)^{\bullet r} -M_{\text{\rm \tiny low}}^{\bullet r}
\overset{\eqref{Ml}}{\leq} 
\bigl(M_{\up}^{\bullet r}\bigr)^{\bullet r} -M_{\text{\rm \tiny low}} \quad\text{on  $D$}.
\end{multline}

Using  Lemma \ref{lem2} we set $f:=f_{\sf Z}g\neq 0$. Then $f({\sf Z})=0$ since ${\sf Z}=\Zero_{f_{\sf Z}}$, and   
\begin{equation}\label{lnf+}
\ln |f|=\ln |f_{\sf Z}g|
=\ln |f_{\sf Z}|+\ln |g|\overset{\eqref{ln}}{\leq} 
 \ln |f_{\sf Z}|+h^{\bullet r}+R\overset{\eqref{lnf}}{\leq}
\bigl(M_{\up}^{\bullet r}\bigr)^{\bullet r} -M_{\text{\rm \tiny low}}+R
\quad\text{on  $D$}.
\end{equation}
By Lemma \ref{lem4} with $u:=M_{\up}$  we obtain 
$\ln |f|\overset{\eqref{lnf+}}{\leq}
M_{\up}^{\odot  \widehat r} -M_{\text{\rm \tiny low}}+R$ on  $D$. 
\end{proof}

\section{Zero subsets  in the complex plane}\label{Sec4}

In this section, we give the results of Theorems \ref{thn} and \ref{th0} 
a form related to subharmonic functions of polynomial growth and point out a very general case when the necessary and sufficient conditions coincide.  We denote by  
\begin{equation*}
{\rm \sf  Pot}:=\Bigl\{ p\in \sbh_* (\CC)\Bigm| \limsup_{z\to \infty} \frac{p(z)}{\ln |z|}<+\infty  
\Bigr\}
\end{equation*} 
the convex cone over $\RR^+$ of all {\it subharmonic functions of polynomial growth\/} \cite[6.7.2]{Hayman}.
We  use the convex subcone over $\RR^+$
\begin{equation}\label{Pot0}
{\rm \sf  Pot}_0^{+1}:=
\left\{p\in {\rm \sf  Pot}\Bigm|  p(0)=0, \quad p\geq 0\text{ on $\CC$}
, \quad \limsup_{z\to \infty} \frac{p(z)}{\ln |z|}\leq 1 \right\}\subset {\rm \sf  Pot}
\end{equation}
  of {\it positive subharmonic functions of polynomial  growth 
with unit upper  seminormization at\/ $\infty$.}

\begin{theorem}\label{th3} Let ${\sf Z}$ be a  locally finite distribution of points in $\CC$ and $0\notin {\sf Z}$.
Let  $M$ be a $d$-subharmonic function \eqref{M} on $D:=\CC$. Suppose that  $M_{\text{\rm \tiny low}}(0)\neq -\infty$
and there are numbers $P\in \RR^+$ and $C\in \RR^+$
such that \eqref{M0} holds. Then the following three statements are equivalent:
\begin{enumerate}[{\rm I.}]
\item\label{I} 
There exists an entire  function $f\neq 0$ such that  $f({\sf Z})=0$ and\/  
\begin{equation}\label{eqC}
\ln \bigl|f(z)\bigr| \leq M(z) \quad\text{at  each $z\in \CC$}.
\end{equation}
\item\label{II}
There is a number  $C\in \RR^+$ such that, for each   $p\in {\rm \sf  Pot}_0^{+1}$, we have
\begin{equation}\label{v0}
\sum_{j} p\Bigl(\frac1{\Bar {\sf z}_j}\Bigr)\leq  \int_{\CC}p\Bigl(\frac1{\Bar  z}\Bigr)\dd \varDelta_M(z)+C
\quad\text{provided $p\Bigl(\frac1{\Bar  z}\Bigr)$ is $\varDelta_M$-summable on $\CC\setminus 0$.}
\end{equation}

\item\label{III} There are numbers  $C\in \RR^+$ and $R_0>0$ such that \eqref{v0} is fulfilled  
for each  
 \begin{subequations}\label{P1}
\begin{align}
p\in {\rm \sf  Pot}_0^{+1}\bigcap C^{\infty}(\CC)\bigcap \har \bigl(\CC\setminus \overline D(0,R_0)\bigr) 
\tag{\ref{P1}P}\label{P1P}
\\
\intertext{such that $p = 0$ on some neighborhood of the origin
and}  
p(z)=\ln |z|+O(1)\quad\text{as $z\to \infty$}.
\tag{\ref{P1}$_0$}\label{p4}
\end{align}
\end{subequations}  
\end{enumerate}
\end{theorem}

\begin{proof} Here \eqref{eqC} is \eqref{q} for $D=\CC$. 
 Let $z_0:=0$.  The inversion transformation
\begin{equation*}
z\underset{z\in \CC}{\longmapsto} \frac{1}{\Bar z}, \quad 0\mapsto \infty\mapsto 0
\end{equation*} 
from  $\CC_{\infty}$  onto $\CC_{\infty}$
gives a bijection from  ${\sf Pot}_0^{+1}$ onto ${\sf PJ}_0^{\uparrow}(\CC)$  and 
a bijection from the class  \eqref{P1} onto 
the class \eqref{inP}, or onto the class $\mathcal V(U_{z_0})$ from \eqref{PJ1}. 
Thus, \eqref{v0} is \eqref{V} for $v(z)\underset{z\in \CC_{\infty}}{=}p(1/\Bar z)$. 

Hence the implication \ref{I}$\Longrightarrow$\ref{II} follows from Theorem \ref{thn}, 
the   implication \ref{II}$\Longrightarrow$\ref{III}  is obvious, and 
 the   implication \ref{III}$\Longrightarrow$\ref{I}  
follows from Theorem \ref{th0}  if we take into account condition \eqref{M0}.
\end{proof}

\begin{acknowledgments}
 The work was supported by a Grant of the Russian Science Foundation
(Project No. 18-11-00002).
\end{acknowledgments}





\begin{thebibliography}{17}

\bibitem{Rans}
Th.~Ransford, ``Potential Theory in the Complex Plane'', Cambridge Univ. Press, Cambridge (1995).

\bibitem{HK} 
W.\,K.~Hayman,  P.\,B. Kennedy, ``Subharmonic functions'', Vol. 1, London Math. Soc. Monogr., {\bf 9}, Academic Press, London\,--\,New York (1976). 

\bibitem{Helms}
L.\,L.~Helms,  ``Introduction to Potential Theory", Wiley Interscience, 
New York\,--\,London\,--\,Sydney\,--\,Toronto (1969).

\bibitem{KhaRoz18}
B.\,N.~Khabibullin, A.\,P.~Rozit,
``On the Distribution of Zero Sets of Holomorphic Functions", Funct. Anal. Appl., {\bf 52}:1, 21--34 (2018).

\bibitem{MenKha19}
E.~B.~Menshikova, B.~N.~Khabibullin, ``On the Distribution of Zero Sets of Holomorphic Functions. II",
Funct. Anal. Appl., {\bf 53}:1, 65--68 (2019).

\bibitem{KhaKha19}
   B.\,N.~Khabibullin, F.\,B.~Khabibullin, ``On the Distribution of Zero Sets of Holomorphic Functions. III. Inversion Theorems",
Funct. Anal. Appl. {\bf 53}:2, 110--123 (2019).

\bibitem{MenKha20}
   E.\,B.~Menshikova, B.\,N.~Khabibullin,
``A criterion for the sequence of roots of holomorphic function with restrictions on its growth",
Russian Mathematics, {\bf 64}:5, 49--55 (2020).

\bibitem{Arsove53} 
M.\,G.~Arsove, ``Functions representable as differences of subharmonic functions'',
Trans. Amer. Math. Soc., {\bf 75}, 327--365  (1953).

\bibitem{Az}
   V.~Azarin, ``Growth Theory  of Subharmonic Functions"
Birkh\"auser,  Basel\,--\,Boston\,--\,Berlin (2009).

\bibitem{Kha03}
   B.\,N.~Khabibullin,  ``Criteria for (sub-)harmonicity and continuation of (sub-)harmonic functions",
Siberian Math. J., {\bf 44}:4, 713--728 (2003).

\bibitem{Beardon}
A.\,F. Beardon, ``Integral means of subharmonic functions", Proc. Camb. Philos. Soc., {\bf 69}, 151--152  (1971).

\bibitem{FM}
P. Freitas, J.\,P. Matos, ``On the characterization of harmonic and subharmonic functions via mean-value properties",
Potential Anal., {\bf 32}:2, 189--200  (2010).

\bibitem{KhaRozKha19}
Khabibullin, B.N., Rozit,  A.P.,  Khabibullina, E.B.: Order versions of the Hahn\,--\,Banach theorem and envelopes. II. Applications to the function theory, (Russian). Complex Analysis. Mathematical Physics, Itogi Nauki i Tekhniki. Ser. Sovrem. Mat. Pril. Temat. Obz., {\bf 162}, VINITI, Moscow, 93--135 (2019); English transl. in Journal of Mathematical Sciences

\bibitem{BaiTalKha16}
 T.\,Yu.~Bayguskarov, G.\,R.~Talipova, B.\,N.~Khabibullin,
``Subsequences of zeros for classes of entire functions of exponential type, allocated by restrictions on their growth",
St. Petersburg Math. J., {\bf 28}:2, 127--151 (2017).

\bibitem{KhaLob}
B.\,N.~Khabibullin, Balayage of Measures with respect to (Sub-)Harmonic Functions, Lobachevskii Journal of Mathematics, {\bf 41}:11, 2179--2189  (2020).


\bibitem{KhaBai16}
  B.~N.~Khabibullin, T.~Yu.~Baiguskarov,  ``The Logarithm of the Modulus of a Holomorphic Function as a Minorant for a Subharmonic Function", Math. Notes, {\bf 99}:4, 576--589 (2016).

\bibitem{Hayman}
W. K. Hayman, ``Subharmonic functions", {\bf 2}, Academic Press, London (1990). 

\end{thebibliography}
\end{document}